\documentclass{SCIYA2017enOL}
\online

\def\supp{\mathop{\text{supp}}}

\def\EE{\mathcal E}\def\FF{\mathcal{F}}
\def\QQQ{\mathbf{Q}}\def\PPP{\mathbf{P}}\def\RRR{\mathbf{R}}\def\EEE{\mathbf{E}}
\def\nnn{\mathbf{n}}

\begin{document}

\ensubject{fdsfd}

\ArticleType{ARTICLES}
\Vol{60}
\No{1}
\BeginPage{1} %

\title[]{On the Dirichlet form of three-dimensional Brownian motion conditioned to hit the origin}
{}

\author[1$\ast$]{FITZSIMMONS Patrick J.}{pfitzsim@ucsd.edu}
\author[2]{LI Liping}{liliping@amss.ac.cn}

\AuthorMark{Fitzsimmons P. J. and Li L.}

\AuthorCitation{Fitzsimmons P. J. and Li L.}

\address[1]{Department of Mathematics, University of California, San Diego,\\ 9500 Gilman Drive La Jolla, California 92093-0112 USA}
\address[2]{RCSDS, HCMS, Academy of Mathematics and Systems Science,\\
Chinese Academy of Sciences, Beijing {\rm100190}, China}

\abstract{Our concern in this paper is the energy form induced by an eigenfunction of a self-adjoint extension of the restriction of the Laplace operator to $C_c^\infty(\mathbf{R}^3\setminus \{0\})$. We will prove that this energy form is a regular Dirichlet form with core $C_c^\infty(\mathbf{R}^3)$. The associated diffusion $X$ behaves like a $3$-dimensional Brownian motion with a mild  radial drift when far  from $0$, subject to an ever-stronger push toward $0$ near that point. In particular $\{0\}$ is not a polar set with respect to $X$. The diffusion $X$ is  rotation invariant, and admits a skew-product representation  before hitting $\{0\}$:  its radial part is a diffusion on $(0,\infty)$ and its angular part is a time-changed Brownian motion on the sphere $S^2$. The radial part of $X$ is a ``reflected''  extension of  the radial part of $X^0$ (the part process of $X$ before hitting $\{0\}$). Moreover,  $X$ is the unique reflecting extension of $X^0$, but  $X$ is not a semi-martingale.}



\maketitle

\section{Introduction}

Our inspiration for this study is the work of Cranston, Koralov, Molchanov, and Vainberg (\cite{MLSB} and \cite{MLSB2}) on a continuous model for homopolymers based on a ``zero-range potential'' perturbation of the 3-dimensional Laplacian. We start by describing probabilistically a  process constructed in \cite{MLSB2}. This process is symmetric (or reversible) with respect to a suitable measure on $\mathbf{R}^3$, and our goal is to investigate it from the point of view of Dirichlet forms. Let $(Z_t)_{t\ge 0}$ be a standard Brownian motion in $\mathbf{R}^3$, with law $\QQQ^x$ when started at $x$. Fix $\gamma>0$, and define {$L^\epsilon_t:=c_\epsilon\cdot\int_0^t 1_{\{|Z_s|\le\epsilon\}}\,ds$}, where $c_\epsilon:={\pi^2\over 8\epsilon^2}+{\gamma\over\epsilon}$. Now ``tilt'' the measure $\QQQ^x_T:=\QQQ^x|_{\FF_T}$, where $\FF_T:=\sigma\{Z_s:0\le s\le T\}$, as follows:
$$
\PPP^{x,\epsilon}_T(B):={\int_B\exp(L^\epsilon_T)\,d\QQQ^x_T\over\int \exp(L^\epsilon_T)\,d\QQQ^x_T },\qquad B\in\FF_T.
$$
Thus paths that spend a lot of  time near $0$ are being heavily weighted under $\PPP^{x,\epsilon}_T$. It is shown in \cite{MLSB2} that as $\epsilon\downarrow 0$ and then $T\to\infty$, the probability measure $\PPP^{x,\epsilon}_T$ converges weakly to the law $\PPP^x$ of a certain ``Brownian motion with singular drift''. It is this perturbed Brownian motion, which we label $X=(X_t)_{t\ge 0}$,  that is the object of our interest. Roughly speaking, with the function $\psi_\gamma$ as defined below in \eqref{PGF}, $X$ is the diffusion on $\RRR^3$ with infinitesimal generator given, for  smooth functions vanishing near $0$, by 
$$
Af={1\over 2}\Delta f+{\nabla\psi_\gamma\cdot\nabla f\over\psi_\gamma}.
$$
(The operator $L_\gamma u:=\psi_\gamma A(\psi_\gamma^{-1}u)-(\gamma^2/2)u$ is a self-adjoint  extension (on $L^2(\RRR^3)$) of ${1\over 2}\Delta|_{C^\infty_c(\RRR^3\setminus\{0\})}$, and is viewed as the result of perturbing the Laplacian ${1\over 2}\Delta$ by a zero-range potential.)

Because the drift $\nabla\log\psi_\gamma(x)=-x(|x|^{-2}+\gamma |x|^{-1})$, $x\not=0$, blows up as $x\to 0$, the process $X$ feels a strong push toward the origin. This push is such that the origin is a regular recurrent point for $X$ (but all other singleton subsets of $\RRR^3$ are polar for $X$). Away from the origin, $X$ behaves  like 3-dimensional Brownian motion with a moderate push toward the origin, but its behavior in time intervals when it visits the origin is so singular that $X$ is not a semimartingale.

We now describe our main results in a little more detail. 
We work on the Hilbert space $L^2(\mathbf{R}^3)$ and start with the unbounded operator 
\begin{equation}\label{LFD}
	L:=\frac{1}{2}\Delta=\frac{1}{2}\sum_{i=1}^3 \frac{\partial^2}{\partial x_i^2}
\end{equation}
with domain
\begin{equation}\label{LFD2}
	\mathcal{D}(L)=C_c^\infty(\mathbf{R}^3_0).
\end{equation}
Here and in the sequel, $\mathbf{R}^3_0:=\mathbf{R}^3\setminus \{0\}$, and $C_c^\infty(\mathbf{R}^3_0)$ denotes the class of all smooth functions with compact support in 
$\mathbf{R}^3_0$. Note that $L$ is symmetric but not self-adjoint on $L^2(\mathbf{R}^3)$. It is an interesting problem to describe the  self-adjoint extensions of   $L$   acting on $L^2(\mathbf{R}^3)$. At the same time,  this is   an important topic in the theory of quantum mechanics; see \cite{SFRH}. For the  following complete characterization to the self-adjoint extensions of $L$ see \cite{SFRH}, \cite{MLSB},  and \cite{PL}.

\begin{lemma}\label{LM1}
The self-adjoint extensions of $L$, to an operator acting on $L^2(\mathbf{R}^3)$, form a one-parameter family $\mathcal{L}_\gamma,\gamma\in \mathbf{R}$. The spectrum of $\mathcal{L}_\gamma$ is given by
\[
	\begin{aligned}
		\text{spec}(\mathcal{L}_\gamma)&=(-\infty,0]\cup \{{\gamma^2/2}\},\quad &\gamma>0,\\
		&= (-\infty, 0],\quad &\gamma\leq 0.
	\end{aligned}
\]
Moreover if $\gamma>0$, then $\gamma^2/2$ is a simple eigenvalue of $\mathcal{L}_\gamma$ with  (normalized) eigenfunction 
\begin{equation}\label{PGF}
	\psi_\gamma(x)=\frac{\sqrt{\gamma}}{\sqrt{2\pi}}\frac{\text{e}^{-\gamma |x|}}{|x|}.
\end{equation}
\end{lemma}

In this paper  we are concerned  with  the \emph{energy form} induced by the eigenfunction $\psi_\gamma$ for a fixed $\gamma>0$. That is
\begin{equation}\label{MFUL}
	\begin{aligned}
		&\mathcal{F}:=\{u\in L^2(\mathbf{R}^3,\psi_\gamma^2dx): \nabla u\in L^2(\mathbf{R}^3,\psi_\gamma^2dx)\}\\
		&\mathcal{E}(u,v):=\frac{1}{2}\int_{\mathbf{R}^3} \nabla u(x)\cdot \nabla v(x) \psi_\gamma(x)^2dx,\quad u,v\in \mathcal{F},
	\end{aligned}
\end{equation}
where $\nabla u$ is the gradient of $u$ in the sense of   distributions. Note that $\psi_\gamma$ is a smooth function on $\mathbf{R}^3_0$ but explodes at $0$. Moreover, $\psi_\gamma\in L^2(\mathbf{R}^3)$ but $\nabla \psi_\gamma\notin L^1_{\text{loc}}(\mathbf{R}^3)$. The explosion of $\psi_\gamma$ at $0$ means that we cannot appeal to  the classical results about  energy forms similar to  \eqref{MFUL}, as found in  \cite{AHS3}, \cite{SFWL}, \cite{MF2}, \cite{MZ}, \cite{MZ2},  \cite{LS}: these results all require {$\nabla \phi\in L^2_{\text{loc}}(\mathbf{R}^3)$}.

However we will see in Theorem~\ref{THM1} that $(\mathcal{E,F})$ is a regular Dirichlet form on $L^2(\mathbf{R}^3,\psi_\gamma^2dx)$ with the core $C_c^\infty(\mathbf{R}^3)$, as a consequence of    Lemma~\ref{LM1}. As a corollary, the operator $-\mathcal{L}_\gamma$, for each $\gamma>0$, is \emph{lower bounded} with parameter $\gamma^2/2$;  see Corollary~\ref{COR1}.  Let $X$ denote the  diffusion process corresponding to $(\mathcal{E,F})$. Then $X$ is an $m$-symmetric recurrent diffusion on $\mathbf{R}^3$ without killing,  where $m(dx) :=\psi_\gamma(x)^2dx$. Moreover, $X$ is recurrent, conservative, and irreducible; see Proposition~\ref{PRO1}. Outside $\{0\}$ the diffusion $X$ is  similar to a 3-dimensional Brownian motion; especially,  each singleton $\{x\}$ is polar, for $x\in\mathbf{R}^3_0$.  But as a result of the explosion of $\psi_\gamma$ at $0$, the $1$-capacity of $\{0\}$ with respect to $(\mathcal{E,F})$ is positive; see Proposition~\ref{PRO2}.  In addition $\{0\}$ is \emph{regular for itself} in the sense that $P^0(T_0=0)=1$ where $P^0$ is the probability measure of $X$ starting from $0$, and $T_0$ is the hitting time of $\{0\}$ with respect to $X$;  see Corollary~\ref{COR3}.

Let {$X^0$} be the part process of $X$ on $\mathbf{R}^3_0$ and $m^0:=m|_{\mathbf{R}^3_0}$. Then $X^0$ is an $m^0$-symmetric diffusion with lifetime $T_0$ on $\mathbf{R}^3_0$, whose associated Dirichlet form is 
\begin{equation}\label{MFUI}
	\begin{aligned}
		&\mathcal{F}^0=\{u\in \mathcal{F}: \tilde{u}(0)=0\}\\
		&\mathcal{E}^0(u,v)=\mathcal{E}(u,v),\quad u,v\in \mathcal{F}^0.
	\end{aligned}
\end{equation}
(Here $\tilde u$ is a quasi-continuous $m$-version of $u$.) It follows from Theorem~3.3.9 of \cite{CM} that $(\mathcal{E}^0,\mathcal{F}^0)$ is regular on $L^2(\mathbf{R}^3_0,m^0)$ with core $C_c^\infty(\mathbf{R}^3_0)$.  Because  $\text{Cap}(\{0\})>0$, the part process $X^0$ is   different from $X$, but the behavior of $X^0$ is easier to understand. Roughly speaking, $X^0$ moves as a diffusion satisfying 
\[
	X_t^0-X_0^0=B_t-\int_0^t \frac{\gamma |X^0_s|+1}{|X^0_s|^2}\cdot X^0_sds
\]
before hitting $\{0\}$, where $(B_t)$ is a three-dimensional Brownian motion.  We can reconstruct $X$ from $X^0$ by stringing together  excursions of $X^0$.  The requisite entrance law $\{\nu_t:t> 0\}$ is uniquely determined by
\[
	\int_0^\infty \nu_tdt=m.
\]

On the other hand, $X$ and $X^0$ are both rotationally-invariant diffusions.  In Proposition~\ref{PRO3} we will write down the  skew product presentation of $X^0$. Its radial part, which dies upon hitting $\{0\}$, is an absorbing diffusion on $(0,\infty)$  and the angular part is a time-changed spherical Brownian motion on $S^2$. In addition the set of  limit points of the angular part of $X^0$,  as time approaches its lifetime, coincides a.s.  with the entire sphere $S^2$.\ 
Unfortunately $X$ does not have an analogous  skew product presentation because the spinning about $0$ of the paths of $X$ at the beginning (and end) of its excursions from $\{0\}$ is too violent. However the radial part of $X$ which is a diffusion on $[0,\infty)$ is simply the  reflection of the radial part of $X^0$. In other words  setting $r_t:=|X^0_t|$ and $\bar{r}_t:=|X_t|$,  we have  
\[
\begin{aligned}
  	&r_t-r_0=\beta_t-\gamma t,\qquad t<T_0, \\
  	&\bar{r}_t-\bar{r}_0=\beta_t-\gamma t+\pi \gamma\cdot  l_t^0,\quad t\geq 0
  	\end{aligned}
\]
where $\beta$ is a 1-dimensional standard Brownian motion and $(l_t^0)_{t\geq 0}$ is the local time of $(\bar{r}_t)_{t\geq 0}$ at $\{0\}$.
We will also examine the Fukushima decomposition of $(\mathcal{E,F})$ with respect to the coordinate functions:
\[
	X_t-X_0=B_t+N_t,\quad t>0
\]
where $B$ is a 3-dimensional standard Brownian motion and $N$ is the zero energy part of $X$. It will be shown in Theorem~\ref{PRO4} that $N$ is \emph{not} of bounded variation. As a corollary, $X$ is not a semi-martingale. By moderating the pole of  $\psi_\gamma$ at $0$ we can generate a sequence of nice semi-martingales that approximate $X$ in a certain sense;  see Proposition~\ref{PRO5}. 

\subsection*{\bf Notation}
For a  domain $E\subset \mathbf{R}^d$, the function classes $C(E),C^1(E)$, and $C^\infty(E)$ are   the continuous functions,   continuously differentiable  differentiable functions, and smooth functions on $E$, respectively. For a Radon measure $\mu$ on $E$, the Hilbert space of all $\mu$-square-integrable functions on $E$ is denoted by $L^2(E,\mu)$ or $L^2(\mu)$, and its norm and inner product will be written as $||\cdot ||_\mu$ and $(\cdot, \cdot)_\mu$. If $\mu$ is the Lebesgue measure, the preceding notation will abbreviated to 
 $L^2(E)$ or $L^2$. Similar notation will be used for  integrable functions.  For any function class $\Theta$, the subclass of all the functions locally  in $\Theta$ (resp. with compact support, bounded) will denoted by $\Theta_{\text{loc}}$ (resp. $\Theta_c,\Theta_b$).

If $(\mathcal{E,F})$ is a Dirichlet form on $L^2(E,\mu)$ and $F$ is a subset of $E$, then the subclass $\mathcal{F}_{F}$ of $\mathcal{F}$ is defined by
\[
\mathcal{F}_F:=\{u\in \mathcal{F}:u=0,\;\mu\text{-a.e. on }F^c\}.
\]
The $\mathcal{E}_1^{\frac{1}{2}}$-norm of $u\in \mathcal{F}$ is 
\[
	||u||_{\mathcal{E}_1^{\frac{1}{2}}}:=[\mathcal{E}(u,u)+(u,u)_\mu]^\frac{1}{2}.
\]
The $1$-capacity of $(\mathcal{E,F})$, as defined in \cite{FU}, will always denoted by $\rm{Cap}$. For basic concepts to do with  Dirichlet forms and associated potential theoretic notions, especially \emph{polar set,  nest, generalized nest, quasi-continuous function, quasi everywhere (q.e. in abbreviation)}, we refer the reader to   \cite{FU}. The quasi-continuous version of   function $u$ is always denoted by $\tilde{u}$.

The diffusion process associated with $(\mathcal{E,F})$ is denoted by $X=(X_t)_{t\ge 0}$. The law of $X$ started at $x\in\RRR^3$ is $\PPP^x$, and the transition semigroup of $X$ is $(P_t)_{t\ge 0}$ defined by $P_tf(x):=\EEE^x[f(X_t)]$ for $t\ge 0$ and $f:\RRR^3\to\RRR$ bounded and measureable. Here $\EEE^x$ is the expectation with respect to $\PPP^x$. 

\section{The Dirichlet forms induced by eigenfunctions}

The general energy form may be defined by the expression
\begin{equation}\label{EUVI}
	\int u(x)v(x)\phi(x)^2dx
\end{equation}
for $u,v\in C_c^\infty(\mathbf{R}^d)$, with respect to some specific function $\phi$ and dimension  $d\ge 1$. The mildest condition  on $\phi$ that we know  ensuring  the closability of $(\EE, C_c^\infty(\mathbf{R}^d))$ on $L^2(\mathbf{R}^d, \phi^2dx)$ is this:  $\phi\in L^2_{\text{loc}}(\mathbf{R}^d)$ and there is a closed set $N$ of Lebesgue measure zero such that the distribution $\nabla \phi$ is in $L^2_{\text{loc}}(\mathbf{R}^d\setminus N)$; see Theorem~2.4 of \cite{AHS3}. In particular $(\mathcal{E}, C_c^\infty(\mathbf{R}^3))$ is closable on $L^2(\mathbf{R}^3,\psi_\gamma^2dx)$ if we choose $N=\{0\}$ in this  condition. Note that $\nabla \psi_\gamma$ is locally bounded on $\mathbf{R}^3_0$ but is not in $L^1_{\text{loc}}(\mathbf{R}^3)$. (Certain other  properties of the diffusion associated with the energy form \eqref{EUVI}, such as the semi-martingale property, require the additional property $\nabla \phi\in L^1_{\text{loc}}(\mathbf{R}^d)$, which $\nabla \psi_\gamma$ does not satisfy.) We begin with a proof of the  regularity of $(\mathcal{E,F})$.

\begin{theorem}\label{THM1}
	Fix $\gamma>0$ and let $\psi_\gamma$ be the eigenfunction given by 
\eqref{PGF}. Then the form \eqref{MFUL} is a regular Dirichlet form on $L^2(\mathbf{R}^3,\psi_\gamma^2(x)dx)$ with  core $C_c^\infty(\mathbf{R}^3)$.
\end{theorem}
\begin{proof}
Since $\psi_\gamma$ is smooth and strictly positive on $\mathbf{R}^3_0$, it follows that $(\mathcal{E,F})$ is a Dirichlet form (not regular) on $L^2(\mathbf{R}^3_0,m^0)$. Hence it is also a Dirichlet form on $L^2(\mathbf{R}^3,m)$. On the other hand clearly
\[
	C_c^\infty(\mathbf{R}^3)\subset \mathcal{F}.
\]
Denote the closure of $(\mathcal{E}, C_c^\infty(\mathbf{R}^3))$ in $(\mathcal{E,F})$ by $(\tilde{\mathcal{E}},\tilde{\mathcal{F}})$. Then $(\tilde{\mathcal{E}},\tilde{\mathcal{F}})$ is a regular Dirichlet form on $L^2(\mathbf{R}^3,m)$. Let $A$ and $\tilde{A}$ be the associated generators of the Dirichlet forms $(\mathcal{E,F})$ and $(\tilde{\mathcal{E}},\tilde{\mathcal{F}})$ respectively. We only need to prove $A=\tilde{A}$.

First we have
\begin{equation}\label{CCIM}
	C_c^\infty(\mathbf{R}^3_0)\subset \mathcal{D}(A)
\end{equation}
and 
\begin{equation}\label{AUFD}
Au=\frac{1}{2}\Delta u+\frac{\nabla \psi_\gamma}{\psi_\gamma}\cdot \nabla u
\end{equation}
for all $u\in C_c^\infty(\mathbf{R}^3_0)$. Indeed for a given  $u\in C_c^\infty(\mathbf{R}^3_0)$, we can deduce that $u\in \mathcal{F}$ and moreover that for all $f\in \mathcal{F}$,
\[
\begin{aligned}
	\mathcal{E}(u,f)&=\frac{1}{2}\sum_{i=1}^3\int\frac{\partial u}{\partial x_i}(x)\frac{\partial f}{\partial x_i}(x)\psi_\gamma(x)^2dx \\
			&=-\frac{1}{2}\sum_{i=1}^3\int f(x) \frac{\partial}{\partial x_i}(\frac{\partial u}{\partial x_i}\psi_\gamma^2)(x)dx \\
			&=-\frac{1}{2}\sum_{i=1}^3\int f(x)\left(\frac{\partial^2 u}{\partial x_i^2}(x)+2 \frac{\partial u}{\partial x_i}(x)\frac{1}{\psi_\gamma(x)}\frac{\partial \psi_\gamma}{\partial x_i}(x)\right)\psi_\gamma^2(x)dx\\
			&=\left(-\frac{1}{2}\Delta u-\frac{\nabla \psi_\gamma}{\psi_\gamma}\cdot \nabla u, f\right)_m.
\end{aligned}
\]
Thus $u\in \mathcal{D}(A)$ and \eqref{AUFD} holds. In the same way we  deduce that 
\begin{equation}\label{IMDA}
	1\in \mathcal{D}(A),\quad A1=0.
\end{equation}
Now define an operator $H$ on $L^2(\mathbf{R}^3)$ by 
\begin{equation}\label{MDHU}
\begin{aligned}
	&\mathcal{D}(H)=\{u\in L^2(\mathbf{R}^3): \psi_\gamma^{-1}\cdot u\in \mathcal{D}(A)\},\\
	&Hu=\psi_\gamma \cdot A(\psi_\gamma^{-1}\cdot u),\quad u\in \mathcal{D}(H).
\end{aligned}
\end{equation}
Then $H$ is a self-adjoint operator on $L^2(\mathbf{R}^3)$. In fact for each $u\in \mathcal{D}(H)$ and $v\in \mathcal{D}(H^*)$, where $H^*$ is the adjoint operator of $H$, we have
\[
	(u,H^*v)_{L^2(\mathbf{R}^3)}=(Hu,v)_{L^2(\mathbf{R}^3)}=(A(\psi_\gamma^{-1}\cdot u), \psi_\gamma^{-1}\cdot v)_m.
\]
It follows that 
\[
	(\psi_\gamma^{-1}\cdot u,\psi_\gamma^{-1}\cdot H^*v)_m=(A(\psi_\gamma^{-1}\cdot u), \psi_\gamma^{-1}\cdot v)_m.
\]
Hence $\psi_\gamma^{-1}\cdot v\in \mathcal{D}(A)$ and $H^*v=\psi_\gamma\cdot A(\psi_\gamma^{-1}\cdot v)$. In other words, $\mathcal{D}(H^*)\subset \mathcal{D}(H)$ and $Hu=H^*u$ for all $u\in \mathcal{D}(H^*)$. On the other hand for all $u,v\in \mathcal{D}(H)$, 
\[
\begin{aligned}
	(Hu,v)_{L^2(\mathbf{R}^3)}&=(A(\psi_\gamma^{-1}\cdot u), \psi_\gamma^{-1}\cdot v)_m=(u, \psi_\gamma\cdot A(\psi_\gamma^{-1}\cdot v))_{L^2(\mathbf{R}^3)}=(u,Hv).
\end{aligned}\]
Hence $\mathcal{D}(H)\subset \mathcal{D}(H^*)$ and $H^*v=Hv$ for all $v\in \mathcal{D}(H)$. Therefore $H$ is self-adjoint on $L^2(\mathbf{R}^3)$. It follows from \eqref{CCIM} and \eqref{MDHU} that 
\[
C_c^\infty(\mathbf{R}^3_0)\subset \mathcal{D}(H)
\]
and for all $u\in C_c^\infty(\mathbf{R}^3_0)$, 
\[
	Hu=\psi_\gamma \cdot A(\psi_\gamma^{-1}\cdot u)=\frac{1}{2}\Delta u-\frac{\gamma^2}{2}u.
\]
Let $H_\gamma:=H+\frac{\gamma^2}{2}$. Then $H_\gamma$ is a self-adjoint extension of the operator $L$ defined  by \eqref{LFD} and  \eqref{LFD2}. On the other hand from \eqref{IMDA} we see that 
\[
	H_\gamma \psi_\gamma=H\psi_\gamma +\frac{\gamma^2}{2}\psi_\gamma=\frac{\gamma^2}{2}\psi_\gamma.
\]
Hence it follows from Lemma~\ref{LM1} that $H_\gamma=\mathcal{L}_\gamma$.

Similarly, we can show that  \eqref{CCIM}, \eqref{IMDA} and \eqref{MDHU} hold for the generator $\tilde{A}$, and then  define an  operator $\tilde{H}$ on $L^2(\mathbf{R}^3)$ as in  \eqref{MDHU}. As before,  $\tilde{H}_\gamma:=\tilde{H}+\gamma^2/2$ is a self-adjoint extension of $L$ and similarly 
\[
	\tilde{H}_\gamma=\mathcal{L}_\gamma=H_\gamma. 
\]
It follows that $A=\tilde{A}$.
\end{proof}

\begin{remark}
	For a given $\gamma>0$ let $\mathcal{L}_\gamma$ be the self-adjoint extension of $L$ as in Lemma~\ref{LM1}. It follows from the above proof that the generator $A$ of $(\mathcal{E,F})$ is  characterized by
	\begin{equation}
	\begin{aligned}
		&\mathcal{D}(A)=\{u\in L^2(\mathbf{R}^3,m):u\cdot \psi_\gamma\in \mathcal{D}(\mathcal{L}_\gamma)\}\\
		&Au=\psi_\gamma^{-1}\cdot \mathcal{L}_\gamma(\psi_\gamma\cdot u)-\frac{\gamma^2}{2}u,\quad u\in \mathcal{D}(A).
	\end{aligned}\end{equation}
The operator $A$ is actually a $\psi_\gamma$-transform of $\mathcal{L}_\gamma$. In other words, let $(P_t)_{t\geq 0}$ and $(Q_t)_{t\geq 0}$ be the semigroups generated by    $A$ and $\mathcal{L}_\gamma$ respectively. Then $(P_t)$ and $(Q_t)$ are symmetric with respect to $m$ and Lebesgue measure, respectively. And for all $u\in L^2(\mathbf{R}^3,m)$, it follows that
\begin{equation}\label{QTUT}
	P_tu=\text{e}^{tA}u=\text{e}^{-\frac{\gamma^2t}{2}}\psi_\gamma^{-1}\cdot \text{e}^{t\mathcal{L}_\gamma}(\psi_\gamma\cdot u)=\text{e}^{-\frac{\gamma^2t}{2}}\psi_\gamma^{-1}\cdot Q_t(\psi_\gamma\cdot u).
\end{equation}
Since $Q_t$ has a (continuous)  density function $q_t(x,y)$ (that is, $Q_t(x,dy)=q_t(x,y)\,dy$; see \cite{SFRH}, \cite{MLSB}, or \cite{PL}),  it follows that
\begin{equation}\label{QTXY}
	P_t(x,dy)=p_t(x,y)m(dy)
\end{equation}
where 
\begin{equation}\label{QTXYF}
	p_t(x,y)=\frac{\text{e}^{-\gamma^2t/2}}{\psi_\gamma(x)\psi_\gamma(y)}q_t(x,y)
	\end{equation}
for $x,y\neq 0$. Recall that \cite{JY} has provided a characterization of $h$-transforms of symmetric Markov processes. The difference here is that neither $\mathcal{L}_\gamma$ nor $\mathcal{L}_\gamma-\gamma^2/2$ is  the generator of Markov process because neither is  Markovian, whereas the transformed operator $A$ is the generator of the Dirichlet form $(\mathcal{E,F})$.
\end{remark}

\begin{corollary}\label{COR1}
Let $\gamma>0$ and the operator $\mathcal{L}_\gamma$ be in Lemma~\ref{LM1}. Then $-\mathcal{L}_\gamma$ is \emph{lower bounded} with parameter $\gamma^2/2$, i.e. 
\[
	(-\mathcal{L}_\gamma u, u)_{L^2(\mathbf{R}^3)}+\frac{\gamma^2}{2}(u,u)_{L^2(\mathbf{R}^3)}\geq 0
\]
for all $u\in \mathcal{D}(\mathcal{L}_\gamma)$. Consequently, the semigroup $(Q_t)$ of $\mathcal{L}_\gamma$ is bounded by $\exp\{\gamma^2\cdot t/2\}$; {\it i.e.} 
\[
	||Q_t||_{L^2(\mathbf{R}^3)}\leq \exp\{{\gamma^2t/2}\}
\]
for all $t\geq 0$.
\end{corollary}

We now record several global  properties of the Dirichlet form $(\mathcal{E,F})$ and the associated diffusion $X$.

\begin{proposition}\label{PRO1}
The Dirichlet form $(\mathcal{E,F})$ is recurrent, conservative, and irreducible. Consequently, the symmetry measure $m$ is an invariant distribution for $X$.
\end{proposition}
\begin{proof}
	It follows from \eqref{IMDA} that $1\in \mathcal{F}$ and $\mathcal{E}(1,1)=0$. Thus $(\mathcal{E,F})$ is recurrent, hence also conservative (that is, $P_t1=1$, $m$-a.e, for all $t>0$).
On the other hand for all $u\in \mathcal{F}$ satisfying $\mathcal{E}(u,u)=0$, it follows that $\nabla u=0, m$-a.e. hence a.e. because $m$ is equivalent to the Lebesgue measure. Thus we can deduce that $u$ is constant $m$-a.e. Because  $m(\mathbf{R}^3)<\infty$ it follows from  Theorem~2.1.11 of \cite{CM} that $(\mathcal{E,F})$ is irreducible. The final assertion follows because  $X$ is conservative, and 
	\[
	m(\mathbf{R}^3)=\int_{\mathbf{R}^3} \psi_\gamma^2(x)dx=1.
	\]
That completes the proof.
\end{proof}

\section{Behavior near $0$}

The process induced by the energy form \eqref{EUVI} is sometimes  called a \emph{distorted Brownian motion}.  In particular, the potential theoretic properties of $(\mathcal{E,F})$ on a given relatively compact open subset $G$ of $\mathbf{R}^3_0$ are equivalent to those of the Brownian motion because the $\mathcal{E}_{G,1}^{1/2}$-norm of the part Dirichlet form $(\mathcal{E}_G,\mathcal{F}_G)$ is equivalent to that of $(\frac{1}{2}D_G, H_0^1(G))$. Here $(\frac{1}{2}D_G, H_0^1(G))$ is the  Dirichlet form of the absorbing Brownian motion on $G$.
However,  because of the singularity of $\psi_\gamma$ at the origin, the process $X$ behaves quite differently from Brownian motion  near the state $0$. It is well known that singletons are  polar sets with respect to 3-dimensional Brownian motion. But $\{0\}$ is not polar with respect to $(\mathcal{E,F})$. Actually $\{0\}$ is the only non-polar singleton with respect to $(\mathcal{E,F})$.

\begin{proposition}\label{PRO2}
 The $1$-capacity of the set $\{0\}$ with respect to $(\mathcal{E,F})$ is positive.
\end{proposition}
\begin{proof}
	Let $B_\epsilon:=\{x\in \mathbf{R}^3 : |x|<\epsilon\}$ for  $\epsilon>0$. Note that the $1$-capacity $\text{Cap}$ satisfies
	\[
		\text{Cap}(\{0\})=\inf_{\epsilon>0}\text{Cap}(B_\epsilon).
	\]
	Hence we need only  compute the capacity of $B_\epsilon$ for each $\epsilon>0$.
	
	Fix $\epsilon>0$, define $f:\mathbf{R}^3\to[0,\infty)$ by
	\[
		f_\epsilon(x):=\left\lbrace 
			\begin{aligned}
			&1,\quad 0\leq |x|\leq \epsilon,\\ 
			&\exp\{c(|x|-\epsilon)\},\quad |x|>\epsilon,
			\end{aligned}
		\right. 
	\]
	where $c=\gamma-\sqrt{\gamma^2+2}<0$. We shall demonstrate that  $f_\epsilon$ is the \emph{1-equilibrium potential} of the set $B_\epsilon$; see \cite{FU}. Things being so  we have
	\[
		\text{Cap}(B_\epsilon)=\mathcal{E}_1(f_\epsilon,f_\epsilon).
	\] 
	
	Firstly,  $f_\epsilon \in \mathcal{F}$. In fact,
	\[
		\begin{aligned}
			\int f_\epsilon(x)^2 \psi_\gamma(x)^2dx&= 4\pi\cdot \frac{\gamma}{2\pi}\int_0^\infty \tilde{f}_\epsilon(r)^2 \frac{\text{e}^{-2\gamma r}}{r^2}r^2dr\\
			&=2\gamma\int_0^\epsilon \text{e}^{-2\gamma r}dr+2\gamma\int_\epsilon^\infty \text{e}^{2(c-\gamma)r-2c\epsilon}dr\\
			&=1+\frac{c}{\gamma-c}\text{e}^{-2\gamma \epsilon}
		\end{aligned}
	\]
	where $\tilde{f}_\epsilon$ is the radial part of $f_\epsilon$. 
	And 
	\[
	\begin{aligned}
		\int |\nabla f_\epsilon|^2(x)\psi_\gamma(x)^2dx&=4\pi\cdot \frac{\gamma}{2\pi}\int_\epsilon^\infty c^2\tilde{f}_\epsilon(r)^2\frac{\text{e}^{-2\gamma r}}{r^2}r^2dr=\frac{\gamma c^2}{\gamma-c}\cdot \text{e}^{-2\gamma \epsilon}.
	\end{aligned}\]
	This shows that  $f_\epsilon\in \mathcal{F}$ and 
	\[
		\mathcal{E}_1(f_\epsilon,f_\epsilon)=1+\frac{c+\gamma c^2}{\gamma-c}\cdot \text{e}^{-2\gamma \epsilon}.
	\]
	In particular, {$(c+\gamma c^2)/(\gamma-c)<0$ implies}
	\[
		\inf_{\epsilon>0}\mathcal{E}_1(f_\epsilon,f_\epsilon)=\lim_{\epsilon\downarrow 0}\mathcal{E}_1(f_\epsilon,f_\epsilon) =\frac{\gamma+\gamma c^2}{\gamma-c}>0.
	\]

Next	it follows from Theorem~2.1.5 of \cite{FU} that to show that $f_\epsilon$ is $1$-excessive we   need to show that 
	\[
		\mathcal{E}_1(f_\epsilon, v)\geq 0
		\]
 for all $v\in \mathcal{F}$ with $\tilde{v}\geq 0$ on $B_\epsilon$. Note that 
	\begin{equation}\label{TFER}
		\tilde{f}''_\epsilon(r)-2\gamma \tilde{f}'_\epsilon(r)-2\tilde{f}_\epsilon(r)=0,\quad r> \epsilon.
	\end{equation}
	We first assume $v$ has a compact support, say $K$. Clearly the weak distribution derivative satisfies
	\[
		\frac{\partial}{\partial x_i}\left(v\psi_\gamma^2 \frac{\partial f_\epsilon}{\partial x_i}\right)=\psi_\gamma^2\frac{\partial v}{\partial x_i}\frac{\partial f_\epsilon}{\partial x_i}+v\frac{\partial}{\partial x_i}\left(\psi_\gamma^2\frac{\partial f_\epsilon}{\partial x_i}\right).
	\]
	Thus
	\[
		\mathcal{E}(f_\epsilon,v)=\frac{1}{2}\sum_{i=1}^3\int \frac{\partial}{\partial x_i}\left(v\psi_\gamma^2 \frac{\partial f_\epsilon}{\partial x_i}\right) \,dx-\frac{1}{2}\sum_{i=1}^3\int v\frac{\partial}{\partial x_i}\left(\psi_\gamma^2\frac{\partial f_\epsilon}{\partial x_i}\right)\,dx.
	\]
	Choose a function $g\in C_c^\infty(\mathbf{R}^3)$ such that $g\equiv 1$ on $K$. Because the support of $u$ is contained in $K$,
	\[
	\begin{aligned}
		\int \frac{\partial}{\partial x_i}\left(v\psi_\gamma^2 \frac{\partial f_\epsilon}{\partial x_i}\right)\,dx&=\int g(x) \frac{\partial}{\partial x_i}\left(v\psi_\gamma^2 \frac{\partial f_\epsilon}{\partial x_i}\right)(x) dx =-\int \frac{\partial g}{\partial x_i}v\psi_\gamma^2 \frac{\partial f_\epsilon}{\partial x_i} dx =0.
	\end{aligned}\]
Hence from \eqref{TFER} we can deduce that
\[\begin{aligned}
	\mathcal{E}_1(f_\epsilon,v)=&\int_{B_\epsilon} f_\epsilon(x)v(x)\psi_\gamma(x)^2dx+ \int_{B_\epsilon^{\text{c}}} v(x)[f_\epsilon(x)\psi_\gamma(x)^2-\frac{1}{2}\sum_{i=1}^3\frac{\partial}{\partial x_i}(\psi_\gamma^2\frac{\partial f_\epsilon}{\partial x_i})(x)]dx\\
	=&\int_{B_\epsilon} f_\epsilon(x)v(x)\psi_\gamma^2(x)dx+\int_{B_\epsilon^{\text{c}}} v(x)\psi_\gamma^2(x)[\tilde{f}_\epsilon(|x|)+\gamma  \tilde{f}'(|x|)-\frac{1}{2}\tilde{f}''(|x|)]dx\\
	=&\int_{B_\epsilon} v(x)\psi_\gamma^2(x)dx\\
	\geq &0.
\end{aligned}\]
Now let $v$ be an arbitrary element of  $\mathcal{F}$ with $\tilde{v}\geq 0$ on $B_\epsilon$. Choose a sequence of functions $\{g_n\}\subset C_c^\infty(\mathbf{R}^3)$ such that $g_n\uparrow 1$ pointwise  and in the norm $||\cdot ||_{\mathcal{E}_1^{1/2}}$. Then $g_n\cdot v\rightarrow v$ in the norm $||\cdot ||_{\mathcal{E}_1^{1/2}}$ {by \cite[Theorem~1.4.2~(ii)]{FU}}, while
\[
	\mathcal{E}_1(f_\epsilon, g_n\cdot v)\geq 0
\]
by the preceding discussion. By letting $n\rightarrow \infty$ we  deduce that 
\[
\mathcal{E}_1(f_\epsilon,v)\geq 0.
\]
We have now shown that $f_\epsilon$ is the $1$-equilibrium potential of $B_\epsilon$, and so
\[
	\text{Cap}(\{0\})=\frac{\gamma+\gamma c^2}{\gamma-c}>0
\]
where $c=\gamma-\sqrt{\gamma^2+2}$. In particular,  $\{0\}$ is not $m$-polar.
\end{proof}

The part process $X^0$ on $\mathbf{R}^3_0$ (that is, $X$ killed on first hitting $0$)  is  an $m^0$-symmetric Markov process whose lifetime is the hitting time $T_0$ of $\{0\}$ with respect to $X$. The associated Dirichlet form $(\mathcal{E}^0,\mathcal{F}^0)$ is given by \eqref{MFUI}. In particular $(\mathcal{E}^0,\mathcal{F}^0)$ is regular with  core $C_c^\infty(\mathbf{R}^3_0)$. Clearly $\mathcal{F}^0$ is a proper subset of $\mathcal{F}$ and we will see in Corollary~\ref{COR4} below that $(\mathcal{E}^0,\mathcal{F}^0)$ is irreducible and transient.  For the following definition we refer the reader to  Definition~7.2.6 of \cite{CM}.

\begin{definition}
A symmetric Hunt process $Y$ is said to be a \emph{reflecting extension} of a symmetric standard process $Y^0$ if the following hold:
\begin{description}
\item[(\textbf{RE.1})] $E$ is a locally compact separable metric space, $m$ is an everywhere dense positive Radon measure on $E$ and $Y$ is an $m$-symmetric Hunt process on $E$ whose Dirichlet form $(\mathcal{E,F})$ on $L^2(E,m)$ is regular.
\item[(\textbf{RE.2})] $Y^0$ is the part process of $Y$ on a non-$\mathcal{E}$-polar, $\mathcal{E}$-quasi-open subset $E^0$ of $E$ whose Dirichlet form $(\mathcal{E}^0,\mathcal{F}^0)$ on $L^2(E_0,m|_{E_0})$ is irreducible. 
\item[(\textbf{RE.3})] $m(F)=0$ where $F=E\setminus E_0$.
\item[(\textbf{RE.4})] The active reflected Dirichlet form $(\mathcal{E}^{0,\text{ref}},(\mathcal{F}^0)_a^{\text{ref}})$ of $(\mathcal{E}^0,\mathcal{F}^0)$ coincides with  $(\mathcal{E,F})$.
\end{description}
\end{definition}
For the definition of the active Dirichlet form see \cite{CZ} {(and also \cite[(6.3.1) and (6.3.2)]{CM})}. 

\begin{theorem}\label{THM2}
	The process $X$ on $\mathbf{R}^3$ corresponding to $(\mathcal{E,F})$ given by \eqref{MFUL} is a reflecting extension of $X^0$ on $\mathbf{R}^3_0$.
\end{theorem}

\noindent We only need to verify (\textbf{RE.4}) in the above definition. 

\begin{lemma}
The active reflected Dirichlet form of $(\mathcal{E}^0,\mathcal{F}^0)$ is  equal to $(\mathcal{E,F})$. Here $(\mathcal{E,F})$ and $(\mathcal{E}^0,\mathcal{F}^0)$ are given by \eqref{MFUL} and \eqref{MFUI} respectively.
\end{lemma}
\begin{proof}
	Since $\psi_\gamma$ is smooth and strictly positive on $\mathbf{R}^3_0$, it follows from Theorems~3.3.1 and 3.3.2 of \cite{FU} that the self-adjoint operator $A$ corresponding to $(\mathcal{E,F})$ on $L^2(\mathbf{R}^3_0,m^0)$ is the maximum element of the class of  \emph{Silverstein extensions} of  the form $\mathcal{E}_S$  defined by
	\[
		\begin{aligned}
			&\mathcal{D}(\mathcal{E}_S):=C_c^\infty(\mathbf{R}^3_0),\\
			&\mathcal{E}_S(u,v):=\mathcal{E}(u,v),\quad u,v\in \mathcal{D}(\mathcal{E}_S).
		\end{aligned}
	\]
(For the definitions of Silverstein extension and the related order, we refer the reader to  \S~3.3  and (3.3.3) of \cite{FU} (or Definition~6.6.1 and 6.6.8 of \cite{CM}).) It follows from Theorem~6.6.9 of \cite{CM} that 
	\[
		(\mathcal{E,F})=(\mathcal{E}^{0,\text{ref}},(\mathcal{F}^0)_a^{\text{ref}}).
	\]
That completes the proof.
\end{proof}

With Theorem~\ref{THM2} in hand we can appeal to {Theorems~6.4.2 and 6.6.10~(ii)} of \cite{CM}  to  deduce the following corollary.

\begin{corollary}
	The extended Dirichlet space  of $(\mathcal{E,F})$ is equal to the reflected Dirichlet space $\mathcal{F}^{0,\text{ref}}$ of $(\mathcal{E}^0,\mathcal{F}^0)$:  
	\begin{equation}\label{MFTE}
	\mathcal{F}_{\text{e}}=\mathcal{F}^{0,\text{ref}}=\{u\in \mathcal{F}^0_{\text{loc}}:\nabla u\in L^2(\mathbf{R}^3,\psi_\gamma^2(x)dx)\},
	\end{equation}
	where $\mathcal{F}^0_{\text{loc}}$ is the class of all the functions $u$ for which  there is an increasing sequence of $\mathcal{E}^0$-quasi-open sets $\{D_n\}$ with $\cup_{n=1}^\infty D_n=\mathbf{R}^3_0$ $\mathcal{E}^0$-q.e. and a sequence $\{u_n\}\subset \mathcal{F}^0$ such that $u=u_n$ a.e. on $D_n$.
\end{corollary}
\begin{remark}
	Note that collection described in \eqref{MFTE} contains all   $u\in L^2_{\text{loc}}(\mathbf{R}^3_0,\psi_\gamma^2(x)dx)$ such that  $\nabla u\in L^2(\mathbf{R}^3,\psi_\gamma^2(x)dx)$. 
\end{remark}

To get at the singular behavior of  $X$ near $\{0\}$, we first examine the skew-product decomposition of the part process  $X^0$.

\begin{proposition}\label{PRO3}
The process $X^0$ admits a \emph{skew-product representation} 
\begin{equation}\label{RTTA}
	(r_t\theta_{A_t})_{t\geq 0}
\end{equation}
 where $(r_t)_{t\geq 0}$ is a symmetric diffusion on $(0,\infty)$, killed at $\{0\}$,  whose speed measure $l$ and scale function $s$ are 
 \[
 \begin{aligned}
 	&l(dx)=2\gamma \text{e}^{-2\gamma x}dx, \\
 	&s(x)=\frac{1}{4\gamma^2}\text{e}^{2\gamma x};
 \end{aligned}\]
$(A_t)_{t\geq 0}$ is the PCAF of $(r_t)_{t\geq 0}$ with Revuz measure is  
\[
\mu_A(dx)=\frac{l(dx)}{x^2}; 
\]
and $\theta$ is a spherical Brownian motion on $S^2:=\{x\in \mathbf{R}^3: |x|=1\}$, and is  independent of $(r_t)_{t\geq 0}$.
\end{proposition}
\begin{proof}
 Let $(r_t)_{t\geq 0}$ be the diffusion on $(0,\infty)$  described in the statement of the theorem, with speed measure $l$ and scale function $s$. Then $(r_t)_{t\geq 0}$ is $l$-symmetric, and it follows from Theorem~3 of \cite{FL} that the associated Dirichlet form on $L^2((0,\infty),l)$ is the closure of 
 \[
 \begin{aligned}
 	\mathcal{D}(\mathcal{E}^{s,l})&= C_c^\infty((0,\infty)),\\
 	\mathcal{E}^{s,l}(u,v)&=\frac{1}{2}\int u'(x)v'(x)l(dx),\quad u,v \in \mathcal{D}(\mathcal{E}^{s,l}).
 \end{aligned}\]
 Note that $\mu_A$ is a positive Radon measure on $(0,\infty)$ with full support. Let $\sigma$ be the normalized surface measure on $S^2$, so that  $\sigma(S^2)=1$. Then $\theta$ is $\sigma$-symmetric, and it follows from Theorem~1.1 of \cite{MY} that the skew product \eqref{RTTA} is an $m:=l\otimes \sigma$-symmetric diffusion on $\mathbf{R}^3_0$ whose Dirichlet form can be expressed as  the closure of 
 \[
 \begin{aligned}
 	\mathcal{D}(\tilde{\mathcal{E}})=&C_c^\infty(\mathbf{R}^3_0),\\
 	\tilde{\mathcal{E}}(u,v)=&\frac{1}{2}\int_{S^2}\int_0^\infty \frac{\partial u}{\partial r}(r,y)\frac{\partial v}{\partial r}(r,y)l(dr) \sigma(dy)+\int_0^\infty \frac{1}{2}(-\Delta_{S^2}u(r,\cdot),v(r,\cdot))_\sigma \,\mu_A(dr)
 \end{aligned}
 \]
 for $u,v\in C_c^\infty(\mathbf{R}^3_0)$,  where $\Delta_{S^2}$ is the \emph{Laplace-Beltrami operator} on $S^2$. It follows from \eqref{AUFD} and 
 \[
 	\Delta f=\frac{1}{r^2}\frac{\partial }{\partial r}\left(r^2\frac{\partial f}{\partial r}\right)+\frac{1}{r^2}\Delta_{S^2}f,\quad f\in C_c^\infty(\mathbf{R}^3)
 \]
 that 
 \[
 \tilde{\mathcal{E}}(u,v)=\mathcal{E}(u,v),\quad u,v\in C_c^\infty(\mathbf{R}^3_0).
 \]
 Hence the closure of $(\tilde{\mathcal{E}},\mathcal{D}(\tilde{\mathcal{E}}))$coincides with  $(\mathcal{E}^0,\mathcal{F}^0)$. 
\end{proof}

\begin{remark}
	In fact $X^0$ is \emph{rotation invariant} in the sense that  for any orthogonal  transformation $T:\mathbf{R}^3\rightarrow \mathbf{R}^3$, $T(X^0)$ has the same distribution as $X^0$. Moreover $(r_t)_{t\geq 0}$ is the radial part of $X^0$, i.e. $r_t=|X^0_t|$ for all $t\geq 0$. Since $s(0+)>-\infty$ and $l$ is a bounded measure, the boundary point  $\{0\}$ is a \emph{regular boundary point} for  $(r_t)_{t\geq 0}$; see  \cite{CM}. In particular $\{0\}$ is not $l$-polar, hence non-polar, with respect to $(r_t)_{t\geq 0}$. In other words, 
	\begin{equation}\label{PXPR}
		\QQQ_r^x(\tau_0<\infty)>0
	\end{equation}
	for each $x\in (0,\infty)$, where $\QQQ_r^x$ is the law of  $(r_t)_{t\geq 0}$ starting from $x$, and $\tau_0:=\inf\{t>0:r_t=0\}$. Note that $\tau_0$ is the lifetime of $(r_t)_{t\geq 0}$. Since $X^0$ is rotation invariant, we  deduce that
	\begin{equation}\label{PXTI}
		\phi(x):=\PPP^x(T_0<\infty)=\QQQ_r^{|x|}(\tau_0<\infty)>0,
	\end{equation}
	for all $x\in \mathbf{R}^3_0$, where $\PPP^x$ is the law of  $X$ starting from $x$, and $T_0:=\inf\{t>0:X_t=0\}$. On the other hand, note that the PCAF $A$ satisfies
	\[
		A_t=\int_0^t \frac{1}{r_s^2}ds,\quad t<\tau_0.
	\]
Because  $(r_t)_{t\geq 0}$ and $\theta$ are independent, it follows that the spherical part $S_t:=\theta_{A_t}$ of $X^0$ is a diffusion on $S^2$ and satisfies the SDE
\[
	dS_t=\frac{1}{r_t}d\theta_t,\quad t<\tau_0.
\]
From an analogue  of Theorem~2.12 of \cite{MAUM}, we can deduce that 
\[
	\int_0^{\tau_0}\frac{1}{r_s^2}ds=\infty,\quad \QQQ_r^x\text{-a.s.}
\]
for all $x\in (0,\infty)$. Since $(r_t)_{t\geq 0}$ and $\theta$ are independent, the set of  limit points of $S_t$ as {$t\uparrow \tau_0$} coincides with the entire sphere $S^2$, a.s.; {\it cf.} \cite{EKB} . As noted by Erickson, this behavior of $X^{0}$ at its lifetime is reflected in the fact that the excursions of $X$ away from $0$  end (and by symmetry begin) with the angular part of $X$ oscillating so violently that each neighborhood of each point of the unit sphere is visited infinitely often. This behavior is the root cause of the fact that $X$ is not a semi-martingale, as is shown in  the next section.
\end{remark}

\begin{corollary}\label{COR3}
	The process $X$ is the unique \emph{one-point extension} of $X^0$ in the sense that $X$ is $m$-symmetric, admits no killing on $\{0\}$, and the part process of $X$ on $\mathbf{R}^3_0$ is $X^0$. Consequently, $\{0\}$ is regular for itself in the sense that $\PPP^0(T_0=0)=1$ where $\PPP^0$ is the law  of $X$ starting from $0$, and $T_0:=\inf\{t>0: X_t=0\}$.
\end{corollary}
\begin{proof}
This  is clear from Theorem~7.5.4 of \cite{CM} and \eqref{PXTI}. 
\end{proof}

\begin{corollary}\label{COR4}
	The Dirichlet form $(\mathcal{E}^0,\mathcal{F}^0)$ is irreducible and transient. 
\end{corollary}
\begin{proof}
	Clearly $(r_t)_{t\geq 0}$ and $\theta$ are both irreducible. It follows from Theorem~7.2 of \cite{MY} and Proposition~\ref{PRO3} that $(\mathcal{E}^0,\mathcal{F}^0)$ is also irreducible. Hence it is  transient because $1\notin \mathcal{F}^0$. 
\end{proof}
 
Unsurprisingly, $(\mathcal{E,F})$ is also rotation invariant.  In fact let $T$ be an orthogonal  transformation from $\mathbf{R}^3$ to $\mathbf{R}^3$. Denote the probability measures, the semigroup and Dirichlet form of $\hat{X}:=T(X)$ by $(\hat{\PPP}^x)_{x\in \mathbf{R}^3}$,  $(\hat{P}_t)_{t\geq 0}$ and $(\hat{\mathcal{E}},\hat{\mathcal{F}})$. Clearly $\hat{X}$ is also $m$-symmetric and for any Borel subset $B\subset \mathbf{R}^3$,
 \[
 	\hat{P}_t 1_B(x)=\hat{\PPP}^x(\hat{X}_t\in B)=\PPP^{T^{-1}x}(X_t\in T^{-1}B)=P_t(1_B\circ T)(T^{-1}x).
 \]
 It follows that for $f\in L^2(\mathbf{R}^3,m)$,
 \[
 	\hat{P}_tf(x)=P_t(f\circ T)(T^{-1}x). 
 \]
 Then 
 \[\begin{aligned}
 	(f-\hat{P}_tf,f)_m&=\int (f\circ T-P_t(f\circ T))(T^{-1}x) f\circ T(T^{-1}x)\psi_\gamma^2(x) dx  \\
 		&=\int (f\circ T-P_t(f\circ T))(y) f\circ T(y)\psi_\gamma^2(y) dTy\\
 		&=(f\circ T-P_t(f\circ T), f\circ T)_m.
 \end{aligned}\]
 Thus $f\in \hat{\mathcal{F}}$ if and only if $f\circ T\in \mathcal{F}$. Moreover
 \[
 	\hat{\mathcal{E}}(f,g)=\mathcal{E}(f\circ T,g\circ T),\quad f,g\in \hat{\mathcal{F}}.
 \]
 From the expression \eqref{MFUL} of $(\mathcal{E,F})$ we can easily deduce that $(\hat{\mathcal{E}},\hat{\mathcal{F}})=(\mathcal{E},\mathcal{F})$. Hence $\hat{X}$ and $X$ have the same distribution. 
 
 Let $(\bar{r}_t)_{t\geq 0}$ be the radial part of $X$. Then $(\bar{r}_t)_{t\geq 0}$ is a diffusion on $[0,\infty)$. Denote its semigroup by $(\bar{q}_t)_{t\geq 0}$. Clearly for any positive function $f$ on $[0,\infty)$ and $r\in [0,\infty)$,
 \[
 \bar{q}_t f(r)=P_t(f\otimes 1_{S^2})(x)
 \]
 for all $x\in \mathbf{R}^3$ such that $|x|=r$. It follows that $(\bar{r}_t)_{t\geq 0}$ is $l$-symmetric and its Dirichlet form is {(see \cite[p.64]{CM})}
 \[
 	\begin{aligned}
 		&\bar{\mathcal{F}}^{s,l}=\{u\in L^2([0,\infty),l):u'\in L^2([0,\infty),l)\}\\
 		&\bar{\mathcal{E}}^{s,l}(u,v)=\frac{1}{2}\int_0^\infty u'(x)v'(x)l(dx), \quad u,v\in \bar{\mathcal{F}}^{s,l}.
 	\end{aligned}
 \]
 Hence $(\bar{r})_{t\geq 0}$ is a reflecting diffusion on $[0,\infty)$ which is reflected at the boundary $\{0\}$ and acts as $(r_t)_{t\geq 0}$ on $(0,\infty)$. On the other hand it follows from Corollary~1 and Theorem~5 of \cite{FL} that
 \begin{equation}\label{RTRB}
 \begin{aligned}
  	&r_t-r_0=B_t-\gamma t,\quad t<\tau_0, \\
  	&\bar{r}_t-\bar{r}_0=B_t-\gamma t+\pi \gamma\cdot  l_t^0
 \end{aligned}\end{equation}
 where $(B_t)_{t\geq 0}$ is a one-dimensional standard Brownian motion and $(l^0_t)_{t\geq 0}$ is the local time of $\bar{r}$ at $\{0\}$;  {\it i.e.},  the PCAF of $(\bar{r}_t)_{t\ge 0}$ with  smooth measure $\delta_{\{0\}}$. In particular, it follows from \eqref{RTRB} that the following corollary holds. Recall that we already have $\PPP^0(T_0=0)=1$ in Corollary~\ref{COR3}.
 
 \begin{corollary}
 For each  $x\in \mathbf{R}^3_0$, 
 \[
 		\phi(x)=\PPP^x(T_0<\infty)= 1.
 \]
 \end{corollary}
 
 Now we can reconstruct the diffusion $X$ by ``stringing together''  its excursions away from  $\{0\}$. The associated It\^o excursion law $\nnn$ is determined by  $X^0$ and a certain $X^0$-\emph{entrance law} {(see \cite[p.43]{Ge})}. A system $\{\nu_t:t>0\}$ of $\sigma$-finite measures on $\mathbf{R}^3_0$ is said to be an $X^0$-entrance law if 
\[
	\nu_sP^0_t=\nu_{s+t}
\]
for every $t,s>0$ where $(P^0_t)_{t\geq 0}$ is the semigroup of $X^0$.
 For the details of constructing a process from excursions  via a suitable entrance law, we refer the reader to \cite{IK},  \cite{STS},  \cite{STS2},  \cite{FT} , \cite{MF},  and \cite{CF}.  Since $X$ admits  no killing inside its state space, it follows that the unique $X^0$-entrance law $\{\nu_t\}$ needed to construct $\nnn$ is characterized by the formula
 \[
 \int_0^\infty \nu_t dt= m.
 \]
Here is a ``skew-product'' description of $\nnn$ that parallels the earlier skew-product decomposition of $X^0$.
On a suitable measure space prepare three independent random objects:

\begin{description}
\item[(a)] A stationary Brownian motion in the unit sphere $S^2$, $(\Theta_t)_{-\infty<t<\infty}$;

\item[(b)] An excursion $(\rho_t)_{0\le t\le \zeta}$ of the process $(\bar{r}_t)_{0\le t<\infty}$ found in \eqref{RTRB};

\item[(c)] A random variable $U$ uniformly distributed on $(0,1)$.
\end{description}

\smallskip

\noindent Now form the time change 

\[
		A(t):=\left\lbrace 
			\begin{aligned}
			&\int_{U\zeta}^t \rho_s^{-2}\,ds,\quad U\zeta\le t<\zeta;\\\ 
			&-\int_t^{U\zeta} \rho_s^{-2}\,ds,\quad 0<t\le U\zeta,
			\end{aligned}
		\right. 
	\]
	
\noindent noticing that $A_{0+}=-\infty$ and $A_{\zeta-}=+\infty$, almost surely. Then the ``distribution'' of the process  
$$
(\rho_t\Theta_{A(t)})_{0<t<\zeta}
$$
is proportional to the the excursion law $\nnn$.  This is consistent with  Erickson's observation that the angular part of the path of our process must ``go wild''
when approaching (or departing) the origin.

\begin{remark}
Although we have limited our discussion to 3-dimensional Brownian motion, a similar development can be made in dimension $d=2$. 
However, there are natural obstructions to our story when $d\ge 4$. Analytically, it is known that the Laplacian restricted to $C^\infty_c(\RRR^d_0)$ admits a unique self-adjoint extension to $L^2(\RRR^d)$, namely the usual Laplacian on $L^2(\RRR^d)$. Thus, the eigenfunction approach taken here appears to be unavailable for $d\ge 4$. Probabilistically, the function $h(x):=|x|^{2-d}$ is harmonic (on $\RRR^d_0$) for the $d$-dimensional Brownian motion. (This corresponds to the limiting case $\gamma=0$ of our construction.) Let $X^*=(X^*_t)_{t\ge 0}$ be the $h$-transform of $d$-dimensional Brownian motion. This is a diffusion on $\RRR^d_0$ with infinitesimal generator ${1\over 2}\Delta f(x)-{(d-2)\over |x|^2}x\cdot \nabla f(x)$. The push toward the origin represented by the drift term in this generator is strong enough that $X^*$ hits the origin with probability $1$ if started away from $0$. But the push is too strong for the process to be able to escape (continuously) from the origin. In fact, if we start $X^*$ uniformly at random on the sphere of radius $\epsilon$ centered at the origin,
normalize by dividing by $\epsilon^{d-2}$, and then send $\epsilon$ to $0$, we obtain a putative excursion law. The resulting measure, call it $\nnn$ would be  the It\^o excursion measure for the recurrent extension of $X^*$, if there were one. But $\nnn$ satisfies
$$
 \nnn[\zeta\in dt ]= C_d\cdot t^{-d/2},\qquad t>0,
$$
where $\zeta$ is the excursion lifetime. In order that it be possible to string together such excursions to obtain a recurrent extension of $X^*$, it is necessary that $\int_0^\infty \min(t,1)\nnn[\zeta\in dt]<\infty$. This latter condition fails for  $d\ge 4$. This would seem to indicate that at least in the radially symmetric case, a \emph{recurrent} distorted Brownian motion that hits the origin is impossible for $d\ge 4$. We hope to explore possible connections between the analytic and probabilistic obstructions in the future.
\end{remark}

 \section{Fukushima's decomposition}
 
Fukushima's decomposition for symmetric Markov processes may be thought of an extension of the familiar semi-martingale decomposition,  is valid even processes of the form $u(X_t)$ ($u\in\mathcal{F}$) that are not semi-martingales. Note that the coordinate functions 
\[
	f^i(x):=x_i,\quad x\in \mathbf{R}^3,i=1,2,3,
\]
are in both $\mathcal{F}$ and $\mathcal{F}^0$.  It follows from Proposition~6 of \cite{FL} that the Fukushima decomposition of these coordinate function relative to $(\mathcal{E}^0,\mathcal{F}^0)$  can be written as
\[
	X_t^0-X_0^0=B_t-\int_0^t \frac{\gamma |X^0_s|+1}{|X^0_s|^2}\cdot X^0_sds, \quad t<T_0
\]
where $B$ is a 3-dimensional standard Brownian motion. Similarly,  for the Dirichlet form $(\mathcal{E,F})$ and the diffusion $X$ there exists a unique \emph{additive functional of zero energy} $N^i$ for each $i=1,2,3$, and a $3$-dimensional standard Brownian motion $B$ such that
\begin{equation}\label{XTXBT}
	X_t-X_0=B_t+N_t,\quad t\geq 0,
\end{equation}
where $N_t=(N^1_t,N^2_t,N^3_t)$; see Theorem 5.5.1 in   \cite{FU}. It follows from Lemma~5.4.4 of \cite{FU} that 
\begin{lemma}
If $t<T_0$, then
\begin{equation}\label{NTIT}
	N_t=-\int_0^t \frac{\gamma |X_s|+1}{|X_s|^2}\cdot X_s\, ds.
\end{equation}
\end{lemma}

On the other hand, recall that the radial parts of $X$ and $X^0$ have the decompositions \eqref{RTRB}. It is a beautiful reflection from $(r_t)_{t\geq 0}$ to $(\bar{r}_t)_{t\geq 0}$. So the natural question is this: Is there an analogous expression relating  $X^0$ and $X$? In other words does $N$ have an expression similar to \eqref{NTIT}, obtained by adding another term with built out of the  local time $(L^0_t)_{t\geq 0}$ of $X$ at $\{0\}$?

An additive functional $A$ is said to be \emph{of bounded variation} if $A_t(\omega)$ is of bounded variation in $t$ on each compact subinterval of $[0,\zeta(\omega))$ for every fixed $\omega$ in the defining set of $A$, where $\zeta$ is the lifetime of $X$. A continuous AF (CAF in abbreviation) $A$ is of bounded variation if and only if $A$ can be expressed as a difference of two PCAFs:
\[
	A_t(\omega)=A^1_t(\omega)-A^2_t(\omega),\quad t<\zeta(\omega), A^1,A^2\in \mathbf{A}^+_c
\]
where $\mathbf{A}_c^+$ is the space of all the PCAFs. Let $\mu_1$ and $\mu_2$ be the Revuz measures of $A^1$ and $A^2$, then 
\[
	\mu_A:=\mu_1-\mu_2
\]
is the  \emph{signed smooth measure} associated with   $A$.  For more details, see \S5.4 of \cite{FU}. We say $N$ in \eqref{XTXBT} is \emph{of bounded variation} if $N^i$ is of bounded variation for $i=1,2,3$.

\begin{theorem}\label{PRO4}
The zero energy part $N$ in \eqref{XTXBT} is not of bounded variation.
\end{theorem}
\begin{proof}
Arguing by contradiction, suppose that  $N$ is of bounded variation. It follows from Theorem~5.5.4 of \cite{FU} that for each $i$ the signed smooth measure $\mu_i$ of $N^i$ satisfies
\[
	\mathcal{E}(f^i,u)=-\langle \mu_i,u\rangle
\]
for all $u\in \mathcal{F}_{b, F_k}$ where $\{F_k\}$ is a generalized nest associated with $\mu_i$;  {\it i.e.},  $\mu_i(F_k)<\infty$ for all $k\geq 1$. Let $F_k^n:=F_k\cap \{x:1/n\leq |x|\leq n\}$ for each $n\geq 1$. Then $F_k^n$ is compact.
For all $u\in C_c^\infty(\mathbf{R}^3)$ with $\supp u\subset F_k^n$, clearly $u\in \mathcal{F}_{b,F_k}$, and
\[
\begin{aligned}
	\langle \mu_i,u\rangle&=-\mathcal{E}(f^i,u)=-\frac{1}{2}\int \frac{\partial u}{\partial x_i} \psi_\gamma^2(x)dx =\frac{1}{2}\int u(x)\frac{\partial \psi_\gamma^2}{\partial x_i}(x)dx =-\int u(x)\frac{\gamma |x|+1}{|x|}\frac{x_i}{|x|}\,m(dx).
\end{aligned}\]
From this we deduce that 
\begin{equation}\label{MIDF}
	\mu_i(dx)=-\frac{\gamma |x|+1}{|x|}\frac{x_i}{|x|}\,m(dx)
\end{equation}
on each $F_k^n$. It follows that \eqref{MIDF} holds on $(\cup_{k\geq 1}F_k)\cap \{x:|x|>0\}$. On the other hand since $(\cup_{k\geq 1}F_k)^c$ is $\mathcal{E}$-polar, we have $\mu_i((\cup_{k\geq 1}F_k)^c)=m((\cup_{k\geq 1}F_k)^c)=0$. Thus \eqref{MIDF} holds on $\{x:|x|>0\}$. Therefore there is a constant $c_i$ such that 
\begin{equation}\label{MIDFG}
	\mu_i(dx)=-\frac{\gamma |x|+1}{|x|}\frac{x_i}{|x|}\,m(dx)+c_i\delta_{\{0\}}.
\end{equation}
In particular $-\frac{\gamma |x|+1}{|x|}\frac{x_i}{|x|}\,m(dx)$ is a signed smooth measure. Consequently, 
\[
	\frac{\gamma |x|+1}{|x|}\frac{|x_i|}{|x|}\,m(dx)
	\]
 is smooth. Since $|x|\leq |x_1|+|x_2|+|x_3|$, it follows that
 \[
 		\frac{\gamma |x|+1}{|x|}\,m(dx)
 \]	
 is also smooth. Then there exists a quasi-continuous and q.e. strictly positive function $g$ such that 
 \begin{equation}\label{IGXF}
 		\int g(x)\frac{\gamma |x|+1}{|x|}\,m(dx)<\infty;
 \end{equation}
 see Thm.~4.22  in  \cite{FPJ2}.
 In particular, it follows from Proposition~\ref{PRO2} that $g(0)>0$. Moreover, because $\{0\}$ is not polar, $g$ is finely continuous at $0$ {by \cite[Theorem~4.2.2]{FU}}. Thus, if we  let $B_{\epsilon}:=\{x: |x|\leq \epsilon\}$ and $T_\epsilon$ the hitting time of $B_{\epsilon}^c$ by $X$, then  (as noted in \cite{MYER}) 
 \[
 	E^0(g(X_{T_\epsilon}))\rightarrow g(0)
 \]
 as $\epsilon\rightarrow 0$, whereas $X_{T_\epsilon}$ is uniformly distributed on $\partial B_{\epsilon}:=\{x:|x|=\epsilon\}$ since $X$ is rotation invariant. Thus 
 \[
 	\int_{S^2}g(\epsilon u)\,\sigma(du)\rightarrow g(0)
 \]
 as $\epsilon \rightarrow 0$. Thus there is a constant $\delta>0$ such that when $\epsilon<\delta$, 
 \[
 \int_{S^2}g(\epsilon u)\, \sigma(du)>\frac{1}{2}g(0).
 \]
 It follows that
 \[\begin{aligned}
 	\int g(x)\frac{\gamma |x|+1}{|x|}m(dx)&=2\pi\gamma\int_0^\infty  \frac{\gamma r+1}{r}\text{e}^{-2\gamma r}dr\int_{S^2} g(ru)\,\sigma(du) \geq 2\pi\gamma\int_0^\delta \frac{\gamma r+1}{r}\text{e}^{-2\gamma r}dr\cdot \frac{1}{2}g(0)
 \end{aligned}\]
which is infinite, in violation of  \eqref{IGXF}. 
\end{proof}

\begin{corollary}
The diffusion $X$ associated with $(\mathcal{E,F})$ is not a semi-martingale.
\end{corollary}

An interesting fact is that the first term in \eqref{MIDFG} is a signed smooth measure with respect to $(\mathcal{E}^0,\mathcal{F}^0)$ but not   with respect to $(\mathcal{E,F})$. The key to this phenomena is of course that  $\{0\}$ is not polar because $\psi_\gamma$ explodes at $0$. But by  modifying $\psi_\gamma$ near $0$ we can obtain a sequence of nice semi-martingale distorted Brownian motions that are  semi-martingales and that approximate   $X$ in a suitable  sense. 

Define $\psi_\gamma^n(x):=\psi_\gamma(x)$ if $|x|\geq 1/n$ and $\tilde{\psi}_\gamma(1/n)$ if $|x|<1/n$, {where $\tilde{\psi}_\gamma$ is the radial function of $\psi_\gamma$.} Then $\psi_\gamma^n$ is a bounded function on $\mathbf{R}^3$ and $\nabla \psi_\gamma^n\in L^2(\mathbf{R}^3)$. Let 
\[
	\begin{aligned}
		&\mathcal{F}^n:=\{u\in L^2(\mathbf{R}^3,\psi_\gamma^n(x)^2dx):\nabla u\in L^2(\mathbf{R}^3,\psi_\gamma^n(x)^2dx)\},\\
		&\mathcal{E}^n(u,v)=\frac{1}{2}\int_{\mathbf{R}^3} \nabla u(x)\cdot \nabla v(x)\psi_\gamma^n(x)^2dx,\quad u,v \in \mathcal{F}^n.
	\end{aligned}
\]
It follows from \cite[Theorem~3.1]{MZ2} (see also \cite{PW})  that $(\mathcal{E}^n,\mathcal{F}^n)$ is regular on $L^2(\mathbf{R}^3,\psi_\gamma^n(x)^2dx)$ with the core $C_c^\infty(\mathbf{R}^3)$.  Note that the associated diffusion $X^n$ of $(\mathcal{E}^n,\mathcal{F}^n)$ has the following Fukushima decomposition with respect to the coordinate functions:
\[
	X^n_t-X^n_0=B_t-\int_0^t \frac{\gamma |X^n_s|+1}{|X^n_s|^2}\cdot X^n_s \cdot 1_{\{|X^n_s|\geq \frac{1}{n}\}} ds.
\]
Notice that $(\mathcal{E}^n,\mathcal{F}^n)$ and $(\mathcal{E,F})$ are defined on different Hilbert spaces. The natural way to relate them is by  $h$-transforms. Recall that $(P_t)_{t\geq 0},(Q_t)_{t\geq 0}$ are the semigroups associated with   $(\mathcal{E,F})$ and  $\mathcal{L}_\gamma$ respectively, and are related by  \eqref{QTUT}. Denote the semigroup of $(\mathcal{E}^n,\mathcal{F}^n)$ by $(Q^n_t)_{t\geq 0}$. Define a semigroup $(Q^n_t)_{t\geq 0}$ on $L^2(\mathbf{R}^3)$ by 
\[
	Q_t^nu:= \text{e}^{\frac{\gamma^2 t}{2}}\psi_\gamma^n\cdot P_t^n(u\cdot (\psi_\gamma^n)^{-1}), \quad u\in L^2(\mathbf{R}^3), t\geq 0.
\]
Then $(\mathcal{E}^n,\mathcal{F}^n)$ is convergent to $(\mathcal{E,F})$ in the following sense:

\begin{proposition}\label{PRO5}
	There is a subsequence $\{n_k\}$  such that $Q_t^{n_k}$ is strongly convergent to $Q_t$ on $L^2(\mathbf{R}^3)$ as $k\rightarrow \infty$ for all $t\geq 0$. 
\end{proposition}

The above proposition follows from Theorem~2.3 of \cite{AKS}.


\Acknowledgements{The second named author is partially supported by {NSFC (No. 11688101 and 11801546) and Key Laboratory of Random Complex Structures and Data Science, Academy of Mathematics and Systems Science, Chinese Academy of  Sciences (No. 2008DP173182)}.}





\end{document}